\documentclass{imsart}
\usepackage{DEF}
\usepackage{graphicx}
\usepackage[colorlinks]{hyperref}
\usepackage{bbm}

\startlocaldefs

\def\d{{\rm{d}}}

\def\n0{n_{0}}
\def\1{\mathbbm{1}}
\endlocaldefs

\begin{document}

\begin{frontmatter}

\title{HOEFFDING'S INEQUALITIES FOR GEOMETRICALLY ERGODIC
MARKOV CHAINS ON GENERAL STATE SPACE\protect\thanksref{T1}}
\runtitle{Hoeffding's inequalities for Markov chains}
\thankstext{T1}{Work partially supported by Polish Ministry of Science and Higher Education  Grants No. N N201387234 and  N N201 608740}

\begin{aug}
\author{\fnms{B{\l}a\.zej} \snm{Miasojedow}
\ead[label=e2]{bmia@mimuw.edu.pl}}

\runauthor{B. Miasojedow}

\affiliation{University of Warsaw}

\address{B. Miasojedow\\Institute of Applied Mathematics and Mechanics\\ University of Warsaw\\ 
Banacha 2, 02-097 Warszawa, Poland\\
\printead{e2}}

\end{aug}

\def\c0{C_{0}}

\begin{abstract}\noindent
We consider Markov chain $X_n$ with spectral gap in $L^2_\pi$ space.  Assume that $f$ is a bounded function on $\mathcal{X}$ with real values. Then the probabilities of large deviations of sums $S_n=\sum_{k=1}^nf(X_k)$ satisfy Hoeffding's-type inequalities. These bounds depend only on the stationary mean $\pi f$, spectral gap  and the end-points of support of $f$. We generalize the results of \cite{león2004optimal} in two directions. In our paper the state space is general and we do not assume reversibility. 
\end{abstract}


\end{frontmatter}

\section{Introduction}
Consider Markov chain  $(X_n)_{n\geq0}$, with values in Polish space $\mathcal{X}$ with Borel $\sigma$-field $\mathcal{B(X)}$ and stationary distribution $\pi$, and a function $f: \mathcal{X}\to [0,1]$. Denote by $\mu=\pi f$ the stationary mean value of $f$. Let $S_n$ be the partial sum of $f(X_n)$, i.e. $S_n=\sum_{k=1}^n f(X_n)$. The main goal of this paper is to derive bounds of probabilities of large deviations for $S_n$. We prove theorems analogous to \cite{león2004optimal} in a more general setting: 
the state space is general and we do not assume reversibility. The following bound is a consequence of our main result:
\begin{thm}\label{ld} If chain $X_n$ is $\psi$-irreducible  and if exists such $\lambda$ that for every function $g$ with $\pi g=0$ the norm  $\left\|Pg\right\|_\pi\leq\lambda$ a following inequality is satisfied
$$\Pr_\nu\left(S_n\geq n(\mu+\varepsilon)\right)\leq\left\|\frac{d\nu}{d\pi}\right\|_\pi\exp\left\{-\frac{1-\lambda}{1+\lambda}\varepsilon^2n\right\}.$$
\end{thm}
Inequalities of this form can play an important role in Monte Carlo Markov chains (MCMC) algorithms because they bound  the error of estimation. Results of this type have been obtained for uniformly ergodic chains in \cite{glynn2002hoeffding} and improved in \cite{kontoyiannis2005relative}. In the case when the state space is discrete related results are obtained by  \cite{Lezaud1998} and  \cite{león2004optimal}. 

We  use a similar technique as in  \cite{león2004optimal}. The first step is to construct an associated chain $X'_n$ and reduce the problem to properties of operator corresponding to $E_\pi \exp(tS'n)$. 
In the second step the problem is reduced to the two-state space case. The main parts of our proof are: Lemma~\ref{l1} which allow us to omit reversibility assumption and Lemma~\ref{glowny} which solves issues introduced by allowing state space to be infinite.

The paper is organized as follows. In Section~\ref{sec:def} we introduce our notations. The main results are established in Section~\ref{sec:main}.  Proofs of key lemmas are postponed in Section~\ref{sec:lemmas}.
\section{Definitions and notation}\label{sec:def}
Throughout this paper $(X_n)_{n\geq 0}$ represents  $\psi$-irreducible Markov chain on a Polish space $\mathcal{X}$ with $\sigma$-field $\mathcal{B(X)}$, transition kernel $P(x,A)$ and stationary distribution $\pi$.  
Recall that Markov chain is $\psi$-irreducible if there exists non-trivial measure $\psi$ such that for all $A\in\mathcal{B(X)}$ with $\psi(A)>0$ and for all $x\in\mathcal{X}$ we have $\mathbb{P}_x(\tau_A<\infty)>0$, where $\tau_A$ is first return to set $A$, i.e. $\tau_A=\inf\{n\geq 1: X_n\in A\}$. 

The linear operator $P$  associated with transition kernel $P(x,A)$ acts to the right on functions and to the left on measures, so that 
$$ Pg(x)=\int_{\mathcal{X}}g(y)P(x,\d y),\quad\quad \nu P(A)=\int_{\mathcal{X}}P(x,A)\nu(\d x).$$
For every measure $\nu$ on $\mathcal{B(X)}$ and every function $g:\mathcal{X}\to\mathbb{R}$ we denote:
$$ \nu g=\int_{\mathcal{X}}g(x)\nu(\d x),\quad\quad g\otimes\nu(x,A)=g(x)\nu(A).$$

\noindent  Consider $P$ as an operator on Hilbert space $L^2_\pi$, the space of functions such that $\pi(f^2)<\infty$,  with inner product $\left\langle f,g\right\rangle= \int_{\mathcal{X}}f(x)g(x)\pi(dx)$.  The norm in $L^2_\pi$ is denoted by $\left\|\cdot\right\|_\pi$. 
As usual, the norm of operator $T$ on $L^2_\pi$ is defined by 
$$\left\|T\right\|_{L^2(\pi)}=\sup_{\{f\ :\ \left\|f\right\|_\pi=1\}}\left\|Tf\right\|_\pi.$$

\goodbreak

\section{Main result}\label{sec:main}
We assume that transition operator admits spectral gap $1-\lambda$ in $L^2(\pi)$, precisely:
\begin{ass}\label{ass:spectral_gap}
\begin{equation*}\left\|P-\Pi\right\|_{L^2(\pi)}=\lambda<1,\end{equation*}
where $\Pi\stackrel{def}{=}1\otimes\pi$.
\end{ass}
\begin{rem}
Note that this assumption for reversible chains is equivalent to existing spectral gap and hence is equivalent to geometric ergodicity \cite{kontoyiannis2009geometric,roberts1997geometric}. In non reversible case exist geometrically ergodic chains, such that Assumption~\ref{ass:spectral_gap} doesn't hold even for any of the $n$-step transition operators \cite{kontoyiannis2009geometric}.
\end{rem}  
Let $f$ be a function from $\mathcal{X}$ to $[0,1]$ and let $S_n$ be a sum $S_n=\sum_{k=1}^nf(X_k)$.

\begin{thm}\label{main} Let $X_n$ be  $\psi$-irreducible Markov chain with stationary distribution $\pi$. Moreover let $1-\lambda$ be a spectral gap of transition kernel $P$. Then the following bounds  hold for all $\varepsilon>0$ such that $\mu+\varepsilon<1$ :
\begin{eqnarray*}\mathbb{P_\pi}\left(S_n\geq n(\mu+\varepsilon)\right)&\leq&
\left[\frac{\mu+\bar{\mu}\lambda}{1-2\frac{\bar{\mu}-\varepsilon}{ 1+\sqrt{\Delta}}}\right]^{n(\mu+\varepsilon)}\left[\frac{\bar{\mu}+\mu\lambda}{1-2\frac{\mu+\varepsilon}{ 1+\sqrt{\Delta}}}\right]^{n(\bar{\mu}-\varepsilon)}\\
&\leq&\exp\left\{-2\frac{1-\lambda}{1+\lambda}\varepsilon^2n\right\},
\end{eqnarray*}
where
$$\Delta=1+\frac{4\lambda(\mu+\varepsilon)(\bar{\mu}-\varepsilon)}{\mu\bar{\mu}(1-\lambda)^2},\quad\bar{\mu}=1-\mu.$$
\end{thm}

\noindent To prove Theorem~\ref{main} we need to consider a new Markov chain $(X'_n)_{n\geq1}$ defined by the transition kernel $Q$ such that for all $A\in\mathcal{B(X)}$,
$$Q(x,A)=(1-\lambda)\pi(A)+\lambda\1(x\in A),$$
where $\lambda$ appear in Assumption~\ref{ass:spectral_gap}.

For any bounded linear operator $T$ on $L^2_\pi$ and any $t\in\mathbb{R}$ define operator $\widehat{T_t}(g)=e^{\frac{t}{2}f}T(e^{\frac{t}{2}f}g)$.
The next Lemma allows us to consider also non-reversible chains. In reversible case we can easily, using similar technique as in \cite{león2004optimal}, replace assumption  \ref{ass:spectral_gap} by weaker one
\begin{equation}\label{eq:vb}
\sup\{z\;:\;z\in\sigma(P-\Pi)\}=\rho<1\;,
\end{equation}
where $\sigma(\cdot)$ denotes a spectrum of operator. In this case we can set $\lambda=\max(0,\rho)$. Condition \eqref{eq:vb} for reversible chains is equivalent to variance bounding property proposed by \cite{robrosvb}. 

\begin{lem}\label{l1} If $S_n$ is defined as above and $P$ satisfies \ref{ass:spectral_gap}, then

 $$\Ex_\pi\exp(tS_n)\leq \left\|\widehat{Q}_t\right\|_{L^2(\pi)}^{n}$$

\end{lem}

As in \cite{león2004optimal} we define a two-state chain. Let $(Y_n)_{n\geq1}$ be a Markov chain on space $\{0,1\}$ with transition matrix with second largest eigenvalue $\lambda\geq 0$ and stationary distribution $\boldsymbol{\mu}:=[1-\mu,\mu]':=[\bar{\mu},\mu]'$. Then transition matrix can be written as 
$$M_{\mu,\lambda}=\lambda\mathbb{I}+(1-\lambda)\boldsymbol{1}\boldsymbol{\mu}'.$$
Let $D^2_t=\diag(1,e^t)$ and denote by $\theta_t$ the Perron-Frobenius eigenvalue of $D^2_t M_{\mu,\lambda}$.
By the same arguments as in Theorem 2 from \cite{león2004optimal} we obtain:
\begin{thm}\label{mp1}Let $X'_n$ have transition kernel $Q$ and let $\mu=\Ex_\pi f(X'_k)$. Then for every convex function $G:\mathbb{R}\to\mathbb{R}$ we have
$$  \Ex_\pi\left[G\left(f(X'_1)+...+f(X'_n)\right)\right]\leq \Ex_\mu\left[G\left(Y_1+...+Y_n)\right)\right],$$
where $Y_n$ is a Markov chain with transition matrix $M_{\mu,\lambda}$.
\end{thm}
We define function $g(x):=\frac{e^{\frac{t}{2}f(x)}}{r_t-\lambda e^{tf(x)}}$. If $r_t$ is solution of equation
\begin{equation}\label{eqr}1=\int_\mathcal{X}\frac{(1-\lambda)e^{tf(x)}}{r_t-\lambda e^{tf(x)}}\pi(\d x)
\end{equation} and if $r_t>\lambda \textrm{ess}\sup_{x\in\mathcal{X}}(e^{tf(x)})$ then function $g$ is positive ($\pi\ a.s.$) and is an eigenfunction of $\widehat{Q}_t$ with eigenvalue $r_t$. Unfortunately the equation (\ref{eqr}) equation for some functions $f$ can have no solution. To avoid this problem lets approximate function $f$ by function with finite number of values. For all $k\in\mathbb{Z^+}$ define $f_k$ as
$$f_k(x)=\sum_{i=1}^k\frac{i}{k}1(x\in A_{i,k}),$$
where  $A_{i,k}:=\{x\in\mathcal{X}\ :\ \frac{i-1}{k}<f(x)\leq\frac{i}{k}\}$. We define $\mu_k$ and $\widehat{Q}_{t,k}$ by replacing $f$ instead $f_k$ in definitions of $\mu$ and $\widehat{Q}_t$ respectively. 
Operator $Q$ is positive (i.e, if $h\geq0$ then $Qh\geq0$) so operator $\widehat{Q}_t$ is positive to and
\begin{eqnarray}
\left\|\widehat{Q}_t\right\|_{L^2(\pi)}&=&\sup_{\{h\geq0\ :\ \left\|h\right\|_\pi=1\}}\left\langle h,\widehat{Q}_th\right\rangle\nonumber \\
&\leq&\sup_{\{h\geq0\ :\ \left\|h\right\|_\pi=1\}}\left\langle h,\widehat{Q}_{t,k}h\right\rangle\nonumber\\
&=&\left\|\widehat{Q}_{t,k}\right\|_{L^2(\pi)}\label{eqq}
\end{eqnarray}
By dominated convergence theorem $\lim_{k\to\infty}\mu_k=\mu.$
Let $\theta_t(x)$ be the Perron-Frobenius eigenvalue of $D^2_t M_{x,\lambda}$. Function $\theta_t(x)$ is continuous so $\theta_{t,k}:=\theta_t(\mu_k)$ converge to $\theta_t$ if $k$ tends to infinity. First we show that (\ref{eqr}) has solution for any function $f_k$.
We consider function $$F(r)=\int_\mathcal{X}\frac{(1-\lambda)e^{tf_k(x)}}{r-\lambda e^{tf_k(x)}}\pi(\d x),$$ for $r>\lambda \textrm{ess}\sup_{x\in\mathcal{X}}(e^{tf_k(x)})$ this function is continuous. If $r$ tends to infinity then $F(r)$ tends to zero. Moreover exist $1\leq j\leq k$ such that  $a:=e^{t\frac{j}{k}}=\textrm{ess}\sup_{x\in\mathcal{X}}(e^{tf_k(x)})$ and $\pi$-measure of set $C_a:=\{x\in\mathcal{X}\ :\ a=e^{tf_k(x)}\}$ is equal to $d>0$. So if $r$ tends to $\lambda a$ then $F(r)$ tends to infinity. Hence exists $r_{t,k}$ such that $F(r_{t,k})=1$ and $g_k(x):=\frac{e^{\frac{t}{2}f(x)}}{r_{t,k}-\lambda e^{tf(x)}}$ is an eigenfunction of $\widehat{Q}_{t,k}$ with eigenvalue $r_{t,k}$

The next Lemma shows that norm of operator $\widehat{Q}_{t,k}$ is equal to largest eigenvalue, which is trivial in finite state space, but in general state space is the main difficulties in proof of Theorem~\ref{main}. 
\begin{lem}\label{glowny} With the above notation for all $k\in\mathbb{Z^+}$ we have the following. 
 \begin{itemize}
  \item[(i)] $$r_{t,k}=\left\|\widehat{Q}_{t,k}\right\|_{L^2(\pi)}$$ 
   \item[(ii)]$$\lim_{n\to\infty}\frac{1}{n}\log \Ex_\pi\exp\left(t\sum_{i=1}^nf_k(X'_i)\right)=\log(r_{t,k})$$
 \end{itemize}
\end{lem}
 
\begin{proof}[Proof of Theorem~\ref{main}.]
By Markov's inequality for all $t>0$ we obtain
\begin{equation}\label{equ1}\mathbb{P_\pi}\left(S_n\geq n(\mu+\varepsilon)\right)\leq e^{-tn(\mu+\varepsilon)}\Ex_\pi e^{tS_n}.
\end{equation}
By Lemma~\ref{l1} and (\ref{eqq}) we have 

$$\Ex_\pi e^{tS_n}\leq\left\|\widehat{Q}_{t}\right\|_{L^2(\pi)}^{n}\leq\left\|\widehat{Q}_{t,k}\right\|_{L^2(\pi)}^{n}.$$
Let $Y^k_1,..Y^k_n$ be a Markov chain with transition matrix $M_{\mu_k,\lambda}$.
From Theorem~\ref{mp1}  and Lemma~\ref{glowny} we obtain
\begin{eqnarray*}\log(\left\|\widehat{Q}_{t,k}\right\|_{L^2(\pi)})&=&\lim_{n\to\infty}\frac{1}{n}\log \Ex_\pi\exp\left(t\sum_{i=1}^nf_k(X'_i)\right)\\&\leq&\lim_{n\to\infty}\frac{1}{n}\log \Ex_\mu\exp\left(t\sum_{i=1}^nf(Y^k_i)\right)\\&=&\log(\theta_{t,k}),
\end{eqnarray*}
therefore
$$\Ex_\pi e^{tS_n}\leq \theta_{t,k}^n.$$
We tend with $k$ to infinity an obtain that
$$\Ex_\pi e^{tS_n}\leq \theta_{t}^n.$$
By Proposition 2 of \cite{león2004optimal} we complete the proof. 

\end{proof}
\begin{cor} With assumptions as in Theorem~\ref{main}. For all $\varepsilon>0$ such that $\mu+\varepsilon\leq 1$ and for all measures $\nu<<\pi$ we have:
$$\mathbb{P}_\nu\left(S_n\geq n(\mu+\varepsilon)\right)\leq\left\|\frac{d\nu}{d\pi}\right\|_p\exp\left\{-2\frac{1-\lambda}{q(1+\lambda)}\varepsilon^2n\right\},$$
where
$$\left\|\frac{d\nu}{d\pi}\right\|_p=\left\{\begin{array}{ll}\left(\int_{\mathcal{X}}\left|\frac{d\nu}{d\pi}\right|^pd\pi\right)^{\frac{1}{p}}& \textrm{if}\ p<\infty\\&\\
ess\sup_{x\in\mathcal{X}}\left|\frac{d\nu}{d\pi}(x)\right|&\textrm{if}\ p=\infty
\end{array}\right.$$
and
$$\frac{1}{p}+\frac{1}{q}=1.$$
\end{cor}
\begin{proof} From H$\ddot{o}$lder's inequality we have

\begin{eqnarray*}\mathbb{P}_\nu\left(S_n\geq n(\mu+\varepsilon)\right)&=&\int_{\mathcal{X}}\mathbb{P}_x\left(S_{n-1}\geq n(\mu+\varepsilon)-f(x)\right)\frac{d\nu}{d\pi}(x)\pi(\d x)\\
&\leq&\left\|\frac{d\nu}{d\pi}\right\|_p\left(\int_{\mathcal{X}}\mathbb{P}_x\left(S_{n-1}\geq n(\mu+\varepsilon)-f(x)\right)^q\pi(\d x)\right)^{\frac{1}{q}}\\
&\leq&\left\|\frac{d\nu}{d\pi}\right\|_p\left(\int_{\mathcal{X}}\mathbb{P}_x\left(S_{n-1}\geq n(\mu+\varepsilon)-f(x)\right)\pi(\d x)\right)^{\frac{1}{q}}\\
&=&\left\|\frac{d\nu}{d\pi}\right\|_p\mathbb{P}_\pi\left(S_n\geq n(\mu+\varepsilon)\right)^{\frac{1}{q}}.
\end{eqnarray*}
\end{proof}

\section{Proofs of key lemmas}\label{sec:lemmas}
\begin{proof}[Proof of Lemma~\ref{l1}]  We know from Cauchy-Schwarz inequality that 
\begin{eqnarray}
 \Ex_\pi\exp(tS_n)&=&\left\langle 1,(e^{tf}P)^n 1\right\rangle = \left\langle e^{\frac{t}{2}f},\widehat{P_t}^{n-1}e^{\frac{t}{2}f}\right\rangle \nonumber
 \\ \label{momgep}
 &\leq& \left\|e^{\frac{t}{2}f}\right\|^2_\pi\left\|\widehat{P_t}\right\|_{L^2(\pi)}^{n-1}=\pi\left(e^{tf}\right)\left\|\widehat{P_t}\right\|_{L^2(\pi)}^{n-1}.
 \end{eqnarray}
For any function $g$ denote its centered version by $g_C:=g-\pi(g)$ and by $P_0:=P-\Pi$.
Since $P$ satisfies \ref{ass:spectral_gap} then by Cauchy-Schwartz inequality we obtain
\begin{eqnarray}\left\|\widehat{P_t}\right\|_{L^2(\pi)} &=& \sup_{\left\{g,h\ :\ \left\|g\right\|_\pi  = \left\|h\right\|_\pi=1\right\}}\left\langle h,\widehat{P_t} g\right\rangle \nonumber \\
&=&\sup_{\left\{g,h\ :\ \left\|g\right\|_\pi=\left\|h\right\|_\pi=1\right\}}\left\langle e^{\frac{t}{2}f}h,Pe^{\frac{t}{2}f} g\right\rangle \nonumber \\
&=&\sup_{ \left\{g,h\ :\ \left\|g\right\|_\pi=\left\|h\right\|_\pi=1\right\}}\left\{\pi(e^{\frac{t}{2}f}h)\pi(e^{\frac{t}{2}f}g)\nonumber \right.\\&&\left.\qquad\qquad\qquad\qquad\qquad\qquad\qquad+\left\langle (e^{\frac{t}{2}f}h)_C,P_0(e^{\frac{t}{2}f} g)_C\right\rangle\right\}\nonumber \\
&\leq&\sup_{ \left\{g,h\ :\ \left\|g\right\|_\pi=\left\|h\right\|_\pi=1\right\}}\left\{\pi(e^{\frac{t}{2}f}h)\pi(e^{\frac{t}{2}f}g)\nonumber \right.\\&&\left.\qquad\qquad\qquad\qquad\qquad\qquad\quad+\lambda\left\|(e^{\frac{t}{2}f}h)_C\right\|_\pi\left\|(e^{\frac{t}{2}f} g)_C\right\|_\pi\right\} \nonumber \\
&\leq&\sup_{ \left\{g,h\ :\ \left\|g\right\|_\pi=\left\|h\right\|_\pi=1\right\}}
\left\{\sqrt{\left[\pi(e^{\frac{t}{2}f}g)\right]^2+\lambda\left\|(e^{\frac{t}{2}f} g)_C\right\|^2_\pi}\nonumber \right.\\&&\left.\qquad\qquad\qquad\qquad\qquad\quad\times\sqrt{\left[\pi(e^{\frac{t}{2}f}h)\right]^2+\lambda\left\|(e^{\frac{t}{2}f} h)_C\right\|^2_\pi}\right\} \nonumber \\
&\leq& \sup_{\left\{g\ :\ \left\|g\right\|_\pi=1\right\}}\left\{\left[\pi(e^{\frac{t}{2}f}g)\right]^2+\lambda\left\|(e^{\frac{t}{2}f} g)_C\right\|^2_\pi\right\}\nonumber \\
&=&\sup_{\left\{g\ :\ \left\|g\right\|_\pi=1\right\}}\left\langle 
e^{\frac{t}{2}f}g,\pi(e^{\frac{t}{2}f}g)+\lambda(e^{\frac{t}{2}f}g)_C\right\rangle\nonumber \\&=&\left\|\widehat{Q_t}\right\|_{L^2(\pi)}.\nonumber
\end{eqnarray}Furthermore we have
\begin{multline*}\pi\left(e^{tf}\right)\left\|\widehat{Q}_t\right\|_{L^2(\pi)}\geq\left\langle e^{\frac{t}{2}f},\widehat{Q}_te^{\frac{t}{2}f}\right\rangle
=\left\langle e^{tf},Qe^{tf}\right\rangle=\pi\left(e^{tf}\right)^2\\+\lambda\pi\left((e^{\frac{t}{2}f})^2_C\right)\geq\pi\left(e^{tf}\right)^2
\end{multline*}
and that completes the proof. 
\end{proof}
\begin{proof}[Proof of Lemma~\ref{glowny}]
Ad. (i) Suppose that $\left\|\widehat{Q}_{t,k}\right\|_{L^2(\pi)}>r_{t,k}$ since $Q$ is self-adjoint operator and $g_k$ is an eigenfunction of this operator  then exists a sequence of function $h_n$ such that following conditions are hold:
\begin{enumerate}
\item for all $n=1,2,...$  $\left\|h_n\right\|_\pi=1$ and
\begin{equation}\label{defh1}h_n\bot g_k\end{equation}
\item sequence of inner products $\left\langle h_n,\widehat{Q}_{t,k}h_n\right\rangle$ is non-decreasing and for all $n=1,2,...$,
\begin{equation}\label{defh2}\left\|\widehat{Q}_{t,k}\right\|_{L^2(\pi)}\leq\left\langle h_n,\widehat{Q}_{t,k}h_n\right\rangle+\frac{1}{n}\end{equation}
\end{enumerate}
Function $g_k$ is positive and 
\begin{equation}\label{eq:1}
C_1=\frac{1}{r_{t,k}}\leq g_k\leq\lambda\textrm{ess}\sup_{x\in\mathcal{X}}g_k(x)=\lambda C_2\qquad \pi\ a.s.
\end{equation}
From (\ref{defh1}) it follows that functions $h^+_n=\max(h_n,0)$ and $h^-_n=\max(-h_n,0)$  satisfy
$$\pi h^+_n\leq\frac{\lambda C_2}{C_1}\pi h^-_n=C_3\pi h^-_n$$
and
\begin{equation}\label{eq:2}
\pi h^-_n\leq\frac{\lambda C_2}{C_1}\pi h^+_n=C_3\pi h^+_n. 
\end{equation}
For all functions $h_n$ 
\begin{eqnarray}\nonumber \left\langle h_n,\widehat{Q}_{t,k}h_n\right\rangle&=&(1-\lambda)\left[\pi(e^{\frac{t}{2}f_k}h_n)\right]^2+\lambda\pi(e^{tf_k}h_n^2)\\ \nonumber \label{eq:3}&=&(1-\lambda)\left[\left(\pi(e^{\frac{t}{2}f_k}h_n^+)\right)^2-2\pi(e^{\frac{t}{2}f_k}h_n^+)\pi(e^{\frac{t}{2}f_k}h_n^-)\right.\\&&\left.\qquad\qquad\qquad\qquad+\left(\pi(e^{\frac{t}{2}f_k}h_n^-)\right)^2\right]+\lambda\pi(e^{tf_k}h_n^2).\end{eqnarray} Operator $\widehat{Q}_{t,k}$ is positive so for all functions $h$ we have  $\left\langle \left|h\right|,\widehat{Q}_{t,k}\left|h\right|\right\rangle\geq\left\langle h, \widehat{Q}_{t,k}h\right\rangle.$ Functions $h^+_n$, $h^-_n$ are positive so from (\ref{defh2})
\begin{eqnarray}\label{eq:4}\nonumber
\frac{1}{n}&\geq
&\left\langle \left|h_n\right|,\widehat{Q}_{t,k}\left|h_n\right|\right\rangle-\left\langle h_n, \widehat{Q}_{t,k}h_n\right\rangle\\\nonumber&=&
4(1-\lambda)\pi(e^{\frac{t}{2}f_k}h_n^+)\pi(e^{\frac{t}{2}f_k}h_n^-)\\&\geq&4(1-\lambda)\pi(h_n^+)\pi(h_n^-)
\end{eqnarray}
Finally from (\ref{eq:2}), (\ref{eq:3}), (\ref{eq:4}) and the definition of $r_{t,k}$ we obtain
 \begin{eqnarray*}\left\langle h_n,\widehat{Q}_{t,k}h_n\right\rangle&\leq&
(1-\lambda)\left[\left(\pi(e^{\frac{t}{2}f_k}h_n^+)\right)^2+\left(\pi(e^{\frac{t}{2}f_k}h_n^-)\right)^2\right]+\lambda\pi(e^{tf_k}h_n^2)\\&\leq
&e^{t}(1-\lambda)\left[(\pi(h_n^+))^2+(\pi (h_n^-))^2\right]+r_{t,k}\\&\leq& \frac{e^{t}C_3}{2n}+r_{t,k}
\end{eqnarray*}
Tending to infinity with n we obtain $\left\|\widehat{Q}_{t,k}\right\|_{L^2(\pi)}\leq r_{t,k}.$ 

Ad. (ii) Since $g_k$ is an eigenfunction with eigenvalue $r_{t,k}$  we have for all $n$
\begin{eqnarray*}\left\|\widehat{Q}_{t,k}\right\|_{L^2(\pi)}&=&r_{t,k}\\&=&\left(\frac{\left\langle g_k,\widehat{Q}_{t,k}^{n-1}g_k\right\rangle }{\left\|g_k\right\|^2_\pi } \right)^{\frac{1}{n-1}}\\&\leq& 
\left(\frac{C_2^2}{\left\|g_k\right\|^2_\pi}\right)^{\frac{1}{n}}\left\langle e^{\frac{t}{2}f_k},\widehat{Q}_{t,k}^{n-1}e^{\frac{t}{2}f_k}\right\rangle  ^{\frac{1}{n-1}}\\
&\leq&
\left(\frac{C_2^2}{\left\|g_k\right\|^2_\pi}\right)^{\frac{1}{n-1}}\left\|\widehat{Q}_{t,k}\right\|_{L^2(\pi)}\\&=&C_4^\frac{1}{n-1}\left\|\widehat{Q}_{t,k}\right\|_{L^2(\pi)}\end{eqnarray*}
but $\Ex_\pi\exp\left(t\sum_{i=1}^nf_k(X'_i)\right)=\left\langle e^{\frac{t}{2}f_k},\widehat{Q}_{t,k}^{n-1}e^{\frac{t}{2}f_k}\right\rangle$. 
\end{proof}
\section*{Acknowledgements}
Author thanks Witold Bednorz, Krzysztof {\L}atu\-szy{\'n}ski and Wojciech Niemiro for helpful comments.
\bibliographystyle{alpha} 

\bibliography{references}

\begin{thebibliography}{KLMM05}

\bibitem[GO02]{glynn2002hoeffding}
P.W. Glynn and D.~Ormoneit.
\newblock {Hoeffding's inequality for uniformly ergodic Markov chains}.
\newblock {\em Statistics \& probability letters}, 56(2):143--146, 2002.

\bibitem[KLMM05]{kontoyiannis2005relative}
I.~Kontoyiannis, L.A. Lastras-Montano, and S.P. Meyn.
\newblock {Relative entropy and exponential deviation bounds for general Markov
  chains}.
\newblock In {\em IEEE, International Symposium on Information Theory}, pages
  1563--1567. IEEE, 2005.

\bibitem[KM12]{kontoyiannis2009geometric}
I.~Kontoyiannis and S.P. Meyn.
\newblock Geometric ergodicity and the spectral gap of non-reversible markov
  chains.
\newblock {\em Probability Theory and Related Fields}, 154(1-2):327--339, 2012.

\bibitem[Lez98]{Lezaud1998}
Pascal Lezaud.
\newblock Chernoff-type bound for finite markov chains.
\newblock {\em Annals of Applied Probability, Vol. 8 (1998), no. 3, pp.
  849--867}, 1998.

\bibitem[LP04]{león2004optimal}
C.A. Le{\'o}n and F.~Perron.
\newblock {Optimal Hoeffding bounds for discrete reversible Markov chains}.
\newblock {\em Annals of Applied Probability}, pages 958--970, 2004.

\bibitem[RR97]{roberts1997geometric}
Gareth~O. Roberts and Jeffrey~S. Rosenthal.
\newblock {Geometric ergodicity and hybrid Markov chains}.
\newblock {\em Electron. Comm. Probab}, 2(2):13--25, 1997.

\bibitem[RR08]{robrosvb}
Gareth~O. Roberts and Jeffrey~S. Rosenthal.
\newblock Variance bounding {M}arkov chains.
\newblock {\em Ann. Appl. Probab.}, 18(3):1201--1214, 2008.

\end{thebibliography}
\end{document}